\documentclass[12pt]{amsart}

\usepackage{latexsym,rawfonts}
\usepackage{amsfonts,amssymb}
\usepackage{amsmath,amsthm}

\usepackage[plainpages=false]{hyperref}
\usepackage{graphicx}
\usepackage{color}

\newtheorem{theorem}{Theorem}
\newtheorem{corollary}{Corollary}

\newtheorem{lemma}{Lemma}

\theoremstyle{definition}

\newcommand{\R}{\mathbb{R}}

\newcommand{\dd}{\mathop{}\!\mathrm{d}}

\newcommand{\set}[1]{\left\{#1\right\}}

\newcommand{\pd}{\partial}
\newcommand{\delbar}{\overline{\nabla}}

\newcommand{\var}{\varphi}

\newcommand{\uS}{\mathbb{S}^{n-1}}

\newcommand{\MA}{Monge-Amp\`ere }

\newcommand{\beq}{\begin{equation}}
\newcommand{\eeq}{\end{equation}}
\newcommand{\beqs}{\begin{eqnarray*}}
\newcommand{\eeqs}{\end{eqnarray*}}
\newcommand{\beqn}{\begin{eqnarray}}
\newcommand{\eeqn}{\end{eqnarray}}

\begin{document}

\title{A flow method for the dual Orlicz-Minkowski problem}

\author{ YanNan Liu \qquad Jian Lu }

\address{YanNan Liu: School of Mathematics and Statistics, Beijing Technology and Business University, Beijing 100048, P.R. China}
\email{liuyn@th.btbu.edu.cn}

\address{Jian Lu: South China Research Center for Applied Mathematics and Interdisciplinary Studies, South China Normal University, Guangzhou 510631, P.R. China}
\email{jianlu@m.scnu.edu.cn}
\email{lj-tshu04@163.com}

\thanks{The authors were supported by Natural Science Foundation of China (11871432). The first author was also supported in part by Beijing Natural Science Foundation (1172005).}

\date{}

\begin{abstract}
In this paper the dual Orlicz-Minkowski problem, a generalization of the $L_p$ dual Minkowski problem, is studied.
By studying a flow involving the Gauss curvature and support function, we obtain a new existence result of solutions to
this problem for smooth measures.
\end{abstract}



\maketitle
\vskip4ex

\section{Introduction}

As important developments of the classical Brunn-Minkowski theory in convex geometry, the
Orlicz-Brunn-Minkowski theory and the dual Orlicz-Brunn-Minkowski theory are
developing rapidly in recent years. They have attracted great attention from
many researchers, see for example
\cite{CY.AiAM.81-2016.78,
  GHW.JDG.97-2014.427, GHWY.JMAA.430-2015.810,
  HSX.MA.352-2012.517, HP.Adv.323-2018.114, HH.DCG.48-2012.281,
  HLX.Adv.281-2015.906,
  HLYZ.Acta.216-2016.325,
HLYZ.JDG.110-2018.1,
  HLYZ.DCG.33-2005.699,
  Kon.AiAM.52-2014.82, Li.GD.168-2014.245, Lud.Adv.224-2010.2346,
  LYZ.JDG.56-2000.111,
  LYZ.PLMS3.90-2005.497, 
  LYZ.Adv.223-2010.220, Mes.JMAA.443-2016.146, SL.Adv.281-2015.1364, XJL.Adv.260-2014.350,
   ZZX.Adv.264-2014.700,
  ZX.Adv.265-2014.132}
and references therein.
In the recent groundbreaking work \cite{HLYZ.Acta.216-2016.325}, Huang, Lutwak,
Yang and Zhang introduced dual curvature measures and found their associated
variational formulas for the first time.
These new curvature measures were proved to connect two different known
measures, namely cone-volume measure and Alexandrov integral curvature.
Then the dual curvature measures were extended into the $L_p$ case in
\cite{LYZ.Adv.329-2018.85}.
The so called $L_p$ dual curvature measures partially unified various measures
in the $L_p$ Brunn-Minkowski theory and the dual Brunn-Minkowski theory.
Very recently, the $L_p$ dual curvature measures were developed to the Orlicz
case \cite{GHW+.CVPDE.58-2019.12, GHXY.2018}, which unified more curvature
measures and were named \emph{dual Orlicz curvature measures}.
These curvature measures are of central importance in the dual
Orlicz-Brunn-Minkowski theory, and the corresponding Minkowski problems are
called the \emph{dual Orlicz-Minkowski problems}.

Let $\varphi : (0,+\infty) \to (0,+\infty)$ and $G : \R^n\backslash 0 \to
(0,+\infty)$ be two given continuous functions.
For a convex body $K\subset \R^n$ with the origin $0\in K$,
the dual Orlicz curvature measure of $K$ is defined as
\begin{equation*}
  \widetilde{C}_{\varphi,G}(K,E) :=
  \frac{1}{n} \int_{\boldsymbol{\alpha}^*_K(E)}
  \frac{\varphi(h_K(\alpha_K(u))) G(\rho_K(u)u) \rho_K^n(u)}{h_K(\alpha_K(u))} \dd u
\end{equation*}
for each Borel set $E\subset S^{n-1}$,
where $\boldsymbol{\alpha}_K^*$ is the reverse radial Gauss image, $\alpha_K$ is
the radial Gauss map, $h_K$ is the support function and $\rho_K$ is the radial
function.
See the next section for more details.
For convenience, we here have used a slightly different notation in the definition of
$\widetilde{C}_{\varphi,G}(K,\cdot)$ from that of
\cite{GHW+.CVPDE.58-2019.12, GHXY.2018}.
The \emph{dual Orlicz-Minkowski problem} asks what are the necessary and
sufficient conditions for a Borel measure $\mu$ on the unit sphere $\uS$ to
be a multiple of the dual Orlicz curvature measure of a convex body $K$. Namely,
this problem is to find a convex body $K\subset \R^n$ such that
\begin{equation} \label{dOMP}
  c \dd \widetilde{C}_{\varphi,G}(K,\cdot)  =\dd\mu \text{ on } \uS
\end{equation}
for some positive constant $c$.

When the Radon-Nikodym derivative of $\mu$ with respect to the spherical measure
on $\uS$ exists, namely $\dd\mu =\frac{1}{n} f \dd x$ for a non-negative integrable function
$f$, the equation \eqref{dOMP} is reduced into
\begin{equation} \label{dOMP-f}
  c\, \varphi(h_K) G(\delbar h_K) \det(\nabla^2h_K +h_KI) =f \text{ on } \uS,
\end{equation}
where $\nabla$ is the covariant derivative with respect to an orthonormal frame
on $\uS$, $I$ is the unit matrix of order $n-1$, and
$\delbar h_K(x) :=\nabla h_K(x) +h_K(x) x$ is the point on $\partial K$ whose
outer unit normal vector is $x\in\uS$.
This is a Monge-Amp\`ere type equation.
When both $\varphi$ and $G$ are constant functions, Eq. \eqref{dOMP} or Eq.
\eqref{dOMP-f} is just the classical Minkowski problem.

When $\varphi(s)=s^{1-p}$ and $G(y)=|y|^{q-n}$,
the dual Orlicz curvature measure
$\widetilde{C}_{\varphi,G}(K,\cdot)$ is reduced into the $(p,q)$-th dual
curvature measure $\widetilde{C}_{p,q}(K,\cdot)$ defined in \cite{LYZ.Adv.329-2018.85}.
Correspondingly, equation \eqref{dOMP} or equation \eqref{dOMP-f} will become into
\begin{equation} \label{LpdMP}
  \dd \widetilde{C}_{p,q}(K,\cdot)  =\dd\mu \text{ on } \uS
\end{equation}
or
\begin{equation} \label{LpdMP-f}
  h_K^{1-p} |\delbar h_K|^{q-n} \det(\nabla^2h_K +h_KI) =f \text{ on } \uS
\end{equation}
respectively.
Here the constant $c$ can be merged into $h_K$ because of the
homogeneousness of $\widetilde{C}_{p,q}$.
Eq. \eqref{LpdMP} or Eq. \eqref{LpdMP-f} is just
the \emph{$L_p$ dual Minkowski problem}.
Many important Minkowski type problems are special cases of the $L_p$ dual
Minkowski problem.
When $q=n$, Eq. \eqref{LpdMP} is the $L_p$ Minkowski problem, which includes the
logarithmic Minkowski problem ($p=0$) and the centroaffine Minkowski problem
($p=-n$).
The $L_p$ Minkowski problem has been extensively studied, see e.g.
\cite{BHZ.IMRNI.2016.1807, BLYZ.JAMS.26-2013.831, BT.AiAM.87-2017.58,
CLZ.TAMS.371-2019.2623,
CW.Adv.205-2006.33,
  HLW.CVPDE.55-2016.117, HLX.Adv.281-2015.906, HLYZ.DCG.33-2005.699,
  JLW.Adv.281-2015.845, JLW.JFA.274-2018.826, JLZ.CVPDE.55-2016.41,
  Lu.SCM.61-2018.511, Lu.JDE.266-2019.4394, LJ.DCDS.36-2016.971,
  LW.JDE.254-2013.983, Lut.JDG.38-1993.131, LYZ.TAMS.356-2004.4359,
  Zhu.Adv.262-2014.909,
Zhu.JFA.269-2015.1070}
and Schneider's book \cite{Schneider.2014}, and corresponding references
therein.
When $p=0$, Eq. \eqref{LpdMP} is the dual Minkowski problem, which was first
introduced by \cite{HLYZ.Acta.216-2016.325} and then studied by
\cite{BHP.JDG.109-2018.411,
  CL.Adv.333-2018.87,
HP.Adv.323-2018.114,
 HJ.JFA, JW.JDE.263-2017.3230,Zha.CVPDE.56-2017.18,
Zha.JDG.110-2018.543}.
Note that Eq. \eqref{LpdMP} also contains the Alexandrov problem \cite{Ale.CRDASUN.35-1942.131}
($p=0$ and $q=0$),
which is the prescribed Alexandrov integral curvature problem.

After the $L_p$ dual Minkowski problem was first proposed in
\cite{LYZ.Adv.329-2018.85}, several existence results of solutions were
established in \cite{BF.JDE.266-2019.7980, CHZ.MA.373-2019.953, HZ.Adv.332-2018.57}.
Huang and Zhao \cite{HZ.Adv.332-2018.57}
proved the existence of solutions for general measure $\mu$ when $p>0$ and
$q<0$, for even $\mu$ when $pq>0$ and $p\neq q$, and for smooth $f$ when $p>q$.
Chen, Huang and Zhao \cite{CHZ.MA.373-2019.953}
studied the smooth case when $f$ is even and $pq\geq0$.
B\"{o}r\"{o}czky and Fodor \cite{BF.JDE.266-2019.7980}
proved the existence of solutions for general measure $\mu$ when $p>1$ and $q>0$.
Their existence results for the non-even case can be stated as the following two
theorems.
\let\xxxOldthetheorem\thetheorem
\def\thetheorem{\Alph{theorem}}
For a general measure $\mu$, we have
\begin{theorem} \label{thmA}
  \textup{(1)} \cite[Theorem 1.2]{HZ.Adv.332-2018.57}.
  Let $p>0$, $q<0$, and $\mu$ be a non-zero finite Borel measure on $\uS$. There
  exists a convex body $K\subset \R^n$ containing the origin in its interior such that Eq.
  \eqref{LpdMP} holds if and only if $\mu$ is not concentrated on any closed hemisphere.

  \textup{(2)} \cite[Theorem 1.2]{BF.JDE.266-2019.7980}.
  Let $p>1$, $q>0$, and $\mu$ be a finite Borel measure on $\uS$ that is not
  concentrated on any closed hemisphere. Then there exists a convex body $K\subset
  \R^n$ containing the origin such that Eq. \eqref{LpdMP} holds when $p\neq q$ or
  holds up to a positive constant when $p=q$.
\end{theorem}
And for a function $f$, there is
\begin{theorem}[{\cite[Theorem 1.5]{HZ.Adv.332-2018.57}}] \label{thmB}
  Suppose $p>q$. For any given positive function $f\in C^\alpha(\uS)$ with $\alpha\in(0,1)$,
  there exists a unique solution $h_K\in C^{2,\alpha}(\uS)$ to Eq. \eqref{LpdMP-f}.
  If $f$ is smooth, then the solution is also smooth.
\end{theorem}

\let\thetheorem\xxxOldthetheorem
\setcounter{theorem}{0}

In this paper we mainly consider the existence of solutions to the dual
Orlicz-Minkowski problem.
Our purpose is to extend Theorems \ref{thmA} and \ref{thmB} to Eq. \eqref{dOMP}
or Eq. \eqref{dOMP-f}.
Very recently, Gardner et al. provided a sufficient condition for the existence
of solutions to Eq. \eqref{dOMP}, which contained Theorem \ref{thmA} (1) as a
special case, see \cite[Theorem 6.4]{GHW+.CVPDE.58-2019.12}.
In the sequel \cite{GHXY.2018}, they obtained another existence result about
Eq. \eqref{dOMP}, which contained Theorem \ref{thmA} (2) as a
special case, see \cite[Theorem 6.3]{GHXY.2018}.
To the best of our knowledge, there is no other existence result about Eq.
\eqref{dOMP} or Eq. \eqref{dOMP-f}.

In this paper we extend Theorem \ref{thmB} to the dual Orlicz-Minkowski problem.
Namely we obtain the following

\begin{theorem} \label{thm1}
  Suppose $\alpha\in(0,1)$, $\varphi\in C^\alpha(0,+\infty)$, and $G\in C^1(\R^n\backslash0)$.
  For any given positive function $f\in C^\alpha(\uS)$ satisfying
  \begin{equation} \label{cond-f}
    \limsup_{s\to+\infty} \bigl[ \varphi(s) \max_{|y|=s} G(y) s^{n-1} \bigr]
    <f<
    \liminf_{s\to0^+} \bigl[ \varphi(s) \min_{|y|=s} G(y) s^{n-1} \bigr],
  \end{equation}
  there exists a solution $h_K\in C^{2,\alpha}(\uS)$ to Eq. \eqref{dOMP-f} with $c=1$.
  If $\varphi$, $G$ and $f$ are smooth, then the solution is also smooth.
\end{theorem}

One can see that Theorem \ref{thm1} applies to the situation when
$\varphi(s)=s^{1-p}$, $G(y)=|y|^{q-n}$, $p>q$, and $f>0$, which recovers Theorem
\ref{thmB}.
In the dual Orlicz-Minkowski problem, the constant $c$ in Eq.
\eqref{dOMP} or Eq. \eqref{dOMP-f} is not known and can not be set to $1$ in general, see
\cite{GHW+.CVPDE.58-2019.12, GHXY.2018}.
The equation of the Orlicz-Minkowski problem also contains a constant $c$ which
is not equal to $1$ in general, see
\cite{HLYZ.Adv.224-2010.2485,
  HH.DCG.48-2012.281,
JL.Adv.344-2019.262}.
Note that in our Theorem \ref{thm1},
we can require $c=1$ in Eq. \eqref{dOMP-f}.
This is an advantage of the flow method used in our proof, compared to the usual
variational method.

Theorem \ref{thmB} was proved by the continuity method in
\cite{HZ.Adv.332-2018.57}, for which the uniqueness of solution to Eq.
\eqref{LpdMP-f} is crucial.
However solutions to the dual Orlicz-Minkowski problem may be not unique in
general. One can easily construct examples when $f$ is a constant function in
Eq. \eqref{dOMP-f}.
We may need to impose strong restrictions on $\varphi$ and $G$ to obtain the
uniqueness, see our discussion Theorem \ref{thm3} for $G(y)=G(|y|)$ and see also 
\cite[Theorem 6.5]{GHXY.2018}.
Due to the lack of uniqueness of Eq. \eqref{dOMP-f} in Theorem \ref{thm1}, the
continuity method is no longer applicable to prove Theorem \ref{thm1}. 
Instead, we use a flow method involving the Gauss curvature and support function.

Let $M_{0} $ be a smooth, closed, uniformly convex hypersurface in Euclidean space
$\R^{n}$ enclosing the origin. Suppose that $M_{0}$ is given by a smooth embedding
$X_{0}: \mathbb{S}^{n-1} \rightarrow \R^{n}$.
In this paper we study the long-time existence and convergence of a family of
hypersurfaces $\{M_{t}\}$ given by smooth maps $X(\cdot , t): \mathbb{S}^{n-1}
\rightarrow \R^{n}$ satisfying the following flow equation:
\begin{equation}\label{floweq}
  \begin{split}
    \frac{\pd X}{\pd t}(x,t) &= -f(\nu)\mathcal{K}
    \frac{\langle X, \nu \rangle}{G(X)\varphi(\langle X, \nu \rangle)}\nu  + X(x,t), \\ 
    X(x,0) &= X_{0}(x),
  \end{split}
\end{equation}
where $\langle \cdot, \cdot\rangle$ is the standard inner product in $\R^{n}$,
$\mathcal{K}$ is the Gauss curvature of the hypersurface $M_{t}$ at $X$, and $\nu $
is the unit outer normal vector of $M_t$ at $X$.

The geometric flow generated by the Gauss curvature was first studied by Firey \cite{Fir.M.21-1974.1}
to describe the shape of a tumbling stone.
Since then, various Gauss curvature flows have been extensively studied, see
e.g.
\cite{And.IMRN.1997.1001,
AGN.Adv.299-2016.174,
BCD.Acta.219-2017.1,
BIS.AP.12-2019.259,
CL,
CW.AIHPANL.17-2000.733,
GL.DMJ.75-1994.79,
Ger.CVPDE.49-2014.471,
GN.JEMSJ.19-2017.3735,
Iva.JFA.271-2016.2133,
Urb.JDG.33-1991.91,
Wan.TAMS.348-1996.4501}
and references therein.
Our flow \eqref{floweq} is inspired by \cite{LSW.JEMS} in which a flow was used to
study the dual Minkowski problem.
However flow \eqref{floweq} is more complicated than that in \cite{LSW.JEMS}, since
it involves two more functions $\varphi$ and $G$. Note $G$ is a function of $X$ other
than $|X|$, which needs more efforts to deal with when estimating the principal
curvatures of $M_t$.

In this paper, we obtain   the long-time existence and convergence of the flow \eqref{floweq}.

\begin{theorem} \label{thm2}
  Let $M_{0}$ be a smooth closed uniformly convex hypersurface in Euclidean space
  $\R^{n}$ enclosing the origin.
  Assume $f, \var, G$ are smooth and satisfy the condition \eqref{cond-f}.
  Then flow \eqref{floweq} has a  smooth solution $M_{t}$ for all time
  $t > 0$, and a subsequence of $M_{t}$ converges in $C^{\infty}$ to a
  positive, smooth, uniformly convex solution to Eq. \eqref{dOMP-f} with $c=1$.
\end{theorem}

This paper is organized as follows.
In section 2, we give some basic knowledge about convex bodies and flow \eqref{floweq}.
In section 3, the
long-time existence of  flow $\eqref{floweq}$ will be proved.
  The main difficulties in deriving long-time existence are  to
  obtain uniform positive upper and lower bounds for support function and
  principal curvatures along the flow. By choosing proper auxiliary functions,
  the bounds of principal curvatures are obtained after delicate computations. Then
  long-time existence  follows by standard arguments.
  In section 4, by considering
  a related geometric functional,  which is decreasing along the flow, we can prove that the smooth solution of flow
  \eqref{floweq} can converge for a subsequence to a smooth solution of Eq.
  \eqref{dOMP-f} with $c=1$. This completes the proof of Theorem \ref{thm2}.
  Then Theorem \ref{thm1} follows by a simple approximation.
  At last, we provide a uniqueness result about \eqref{dOMP-f} in a special case.

  \section{Preliminaries}

\subsection{Convex body}

Let $\R^n$ be the $n$-dimensional Euclidean space, and $\uS$ be the unit sphere
in $\R^n$. 
Assume $M$ is a smooth closed uniformly convex hypersurface in $\R^{n}$.
Without loss of generality, we may assume that $M$ encloses the origin.
The support function $h$ of $M$ is defined as
\begin{equation*}
h(x) := \max_{y\in M} \langle y,x \rangle, \quad x\in\uS.
\end{equation*}
And the radial function $\rho$ of $M$ is given by
\begin{equation*}
\rho(u) :=\max\set{\lambda>0 : \lambda u\in M}, \quad u\in\uS.
\end{equation*}
We see that $\rho(u)u\in M$.
Denote by $\nu_M$ the Gauss map of $M$.
Then the radial Gauss map $\alpha_M$ is defined as
\begin{equation*}
\alpha_M(u) := \nu_M(\rho(u)u) \text{ for } u\in\uS.
\end{equation*}
And the reverse radial Gauss map $\alpha_M^*$ is given by
\begin{equation*}
\alpha_M^*(x) :=\set{u\in\uS : \rho(u)u \in \nu_M^{-1}(x)}.
\end{equation*}

For any Borel set $E\subset\uS$ and any bounded integrable function
$g:\uS\to\R$, we have
\begin{equation}
  \label{eq:1}
  \int_{\alpha_M^*(E)} g(u) \dd u
  =\int_{E} g(\alpha_M^*(x)) h(x) \rho(\alpha_M^*(x))^{-n} \det(\nabla^2h+hI)(x) \dd x.
\end{equation}
Note that
\begin{equation*}
\rho(u)u =(\delbar h)(\alpha_M(u)), \quad u\in\uS.
\end{equation*}
Then we easily see that equations \eqref{dOMP} and \eqref{dOMP-f} are equivalent.
One can consult \cite{HLYZ.Acta.216-2016.325} for more details.

\subsection{Curvature flow}

Suppose that $M$ is parametrized by the inverse Gauss map $X :
\mathbb{S}^{n-1}\rightarrow M$,
namely $X(x) =\nu_M^{-1}(x)$.
The support function $h$ of $M$ can be computed by
 \beqn \label{h} h(x) = \langle x, X(x)\rangle, \indent x \in \mathbb{S}^{n-1}.\eeqn
 Here one can see $x$ is the outer normal of $M$ at $X(x)$.
  Let $e_{ij}$ be the standard metric of the sphere $\mathbb{S}^{n-1}$ and $\nabla$ be the gradient on $\mathbb{S}^{n-1}$.
  Differentiating $\eqref{h}$, we have
  \beqs \nabla_{i} h = \langle \nabla_{i}x, X(x)\rangle + \langle x, \nabla_{i}X(x)\rangle.\eeqs
  Since $\nabla_{i}X(x)$ is tangent to $M$ at $X(x)$, we have
  \beqs \nabla_{i} h = \langle \nabla_{i}x, X(x)\rangle. \eeqs
  It follows that
\begin{equation}\label{Xh}  X(x) = \nabla h + hx.\end{equation}

By differentiating $\eqref{h}$ twice, the  second fundamental form $A_{ij}$   of $M$
can be computed in terms of the support function, see for example
\cite{Urb.JDG.33-1991.91},
\beqn \label{A} A_{ij} =  \nabla_{ij}h + he_{ij},\eeqn
where $\nabla_{ij} = \nabla_{i}\nabla_{j}$ denotes the second order covariant derivative with respect to $e_{ij}$.
The  induced metric matix $g_{ij}$ of $M$ can be derived by Weingarten's formula,
\beqn\label{g} e_{ij} = \langle \nabla_{i}x, \nabla_{j}x\rangle  = A_{ik}A_{lj}g^{kl}.  \eeqn
The principal radii of curvature are the eigenvalues of the matrix $b_{ij} =
A^{ik}g_{jk}$.
When considering a smooth local orthonormal frame on $\uS$, by virtue of
\eqref{A} and \eqref{g}, we have
\beqn\label{radii} b_{ij} = A_{ij} = \nabla_{ij}h + h\delta_{ij}.\eeqn
We will use
$b^{ij}$ to denote the inverse matrix of $b_{ij}$.
The Gauss curvature of $X(x) \in M$ is given by
\beqs \mathcal{K}(x) = (\det{(\nabla_{ij}h + h\delta_{ij})})^{-1}.\eeqs

From the evolution equation of $X(x,t)$ in $\eqref{floweq}$, we derive the
evolution equation of the corresponding support function $h(x,t)$:
\begin{equation}\label{seq}
  \frac{\pd h(x,t)}{\pd t}
  = -f(x)\mathcal{K}(x,t)\frac{h}{G(\nabla h + hx)}\frac{1}{\varphi(h)} + h(x,t), \indent x \in \mathbb{S}^{n-1}.
\end{equation}
A self-similar solution of this flow is just equation \eqref{dOMP-f}.

Note that $M$ can be represented by its radial function, namely
\beqs M = \{\rho(u)u: u \in \mathbb{S}^{n-1}\}.\eeqs
It is easy to see that the outer normal vector $\nu = x(u)$ at $\rho(u)u \in M$ is given by
\begin{equation*}
  \nu = \frac{\rho(u)u - \nabla \rho}{\sqrt{\rho^{2} + |\nabla \rho|^{2}}}.
\end{equation*}
And the induced metric and the second fundamental form of $M$ can also be
computed  according to the radial function, see for example
\cite{Ger.CVPDE.49-2014.471}. 

Denote the radial function of $M_t$ by $\rho(u,t)$.
From $\eqref{Xh}$, $u$ and $x$ are related by
\beqn\label{rs} \rho(u)u = \nabla h(x) + h(x)x.
\eeqn
Let $x = x(u,t)$, by $\eqref{rs}$, we have
\beqs \log \rho(u,t) = \log h(x,t) - \log x\cdot u. \eeqs
Differentiating the above identity, we have
\begin{equation*}
\begin{split}
 \frac{1}{\rho(u,t)}\frac{\pd \rho(u,t)}{\pd t}
&= \frac{1}{h(x,t)}(\nabla h\cdot \frac{\pd x(u,t)}{\pd t} + \frac{\pd h(x,t)}{\pd t}) - \frac{1}{u \cdot x}u \cdot \frac{\pd x(u,t)}{\pd t}\\
&= \frac{1}{h(x,t)}\frac{\pd h(x,t)}{\pd t} + \frac{1}{h(x,t)}(\nabla h - \rho(u,t)u)\cdot \frac{\pd x(u,t)}{\pd t}\\
&= \frac{1}{h(x,t)}\frac{\pd h(x,t)}{\pd t}.
\end{split}
\end{equation*}
The evolution equation of  radial function then follows from $\eqref{seq}$,
\begin{equation}\label{req}
\frac{\pd \rho(u,t)}{\pd t} = -f\mathcal{K}(u,t)\frac{\rho}{G(\rho, u)}\frac{1}{\varphi(h)} + \rho(u,t),
\end{equation}
where $\mathcal{K}(u,t)$ denotes the Gauss curvature at $\rho(u,t)u \in M_{t}$ and $f$ takes value at the normal vector $x(u,t) $.
Here and after, we sometimes also write $G(X) =G(\rho,u)$ for convenience.

\section{Long-time existence of the flow }
In this section, we will give a priori estimates about support function and obtain the long-time existence of the flow $\eqref{floweq}$.

We begin with the estimate of the upper and lower bounds of $h$.
\begin{lemma}\label{lem3.1} Let $h$ be a smooth solution of $\eqref{seq}$ on $\mathbb{S}^{n-1} \times [0, T)$, and $f , \var, G$ are smooth functions satisfying $\eqref{cond-f}$,  then
\beqs  \frac{1}{C} \leq  h(x,t) \leq C, \eeqs
where $C$ is a positive constant depending on $\max\limits_{\mathbb{S}^{n-1}}
h(x,0)$, $\min\limits_{\mathbb{S}^{n-1}} h(x,0)$,
$\max\limits_{\mathbb{S}^{n-1}} f(x)$,
$\min\limits_{\mathbb{S}^{n-1}} f(x)$, $ \limsup\limits_{s\to+\infty} \bigl[ \varphi(s) \max\limits_{|y|=s} G(y) s^{n-1} \bigr]$ and $ \liminf\limits_{s\to0^{+}} \bigl[ \varphi(s) \max\limits_{|y|=s} G(y) s^{n-1} \bigr]$.
\end{lemma}
\begin{proof}
Suppose $\max\limits_{\mathbb{S}^{n-1}}h(x,t)$ is attained at $x^{0} \in \mathbb{S}^{n-1}$, then at $x^{0}$, we have
\begin{equation*} \nabla h =0,\indent~~~~\nabla^{2} h \leq 0, \indent \rho = h,
\end{equation*}
and $$ \nabla^{2} h + hI \leq hI.$$
So at $x^{0}$,
\begin{equation*}
\begin{split}
\frac{\pd h}{\pd t}
& =  -f(x)\mathcal{K}(x)\frac{h}{G(\rho, u)}\frac{1}{\varphi(h)} + h\\
& \leq -f(x)\frac{1}{G(h, u)h^{n-1}}\frac{h}{\varphi(h)} + h\\
& = \frac{1}{G(h, u)h^{n-2}\varphi(h)}(G(h, u)h^{n-1}\varphi(h) - f(x))
\end{split}
\end{equation*}
Let $ \overline{A} = \limsup\limits_{s\to+\infty} \bigl[ \varphi(s)
\max\limits_{|y|=s} G(y) s^{n-1} \bigr]$. By $\eqref{cond-f}$, $\varepsilon =
\frac{1}{2}(\min\limits_{\mathbb{S}^{n-1}} f(x) - \overline{A})$ is positive and
there exists a positive constant  $C_{1} > 0 $ such that $$ G(h, u)h^{n-1}\varphi(h) < \overline{A} + \varepsilon $$ for $ h > C_{1}. $ It follows  that
\beqs G(h, u)h^{n-1}\varphi(h)  - f(x)  < \overline{A} + \varepsilon   - \min_{\mathbb{S}^{n-1}} f(x) < 0,\eeqs
from which we have \beqs \frac{\pd h}{\pd t} <0 \eeqs at maximal points.
So \beqs h \leq \max\{C_{1}, \max_{\mathbb{S}^{n-1}} h(x,0)\}.\eeqs

Similarly, we can estimate $\min\limits_{\mathbb{S}^{n-1}}h(x,t)$.
Suppose $\min_{\mathbb{S}^{n-1}}h(x,t)$ is attained at $x^{1}$, then at this point,
\begin{equation*} \nabla h =0,\indent~~~~\nabla^{2} h \geq 0, \indent \rho = h,
\end{equation*}
and $$ \nabla^{2} h + hI \geq hI.$$
And at $x^{1}$,
\begin{equation*}
\begin{split}
\frac{\pd h}{\pd t}
& =  -f(x)\mathcal{K}(x)\frac{h}{G(\rho, u)}\frac{1}{\varphi(h)} + h\\
& \geq -f(x)\frac{1}{G(h, u)h^{n-1}}\frac{h}{\varphi(h)} + h\\
& = \frac{1}{G(h, u)h^{n-2}\varphi(h)}(G(h, u)h^{n-1}\varphi(h) - f(x)).
\end{split}
\end{equation*}
Let $ \underline{A} = \liminf\limits_{s\to0^{+}} \bigl[ \varphi(s)
\max\limits_{|y|=s} G(y) s^{n-1} \bigr]$. By $\eqref{cond-f}$, $\varepsilon =
\frac{1}{2}(\underline{A} - \max\limits_{\mathbb{S}^{n-1}} f(x))$ is positive
and   there exists some positive constant  $C_{2} > 0 $ such that $$ G(h, u)h^{n-1}\varphi(h) > \underline{A}  - \varepsilon $$ for $ h < C_{2}. $ It follows  by $\eqref{cond-f}$
\beqs G(h, u)h^{n-1}\varphi(h)  - f(x)  > \underline{A}  - \varepsilon   - \max_{\mathbb{S}^{n-1}} f(x) > 0,\eeqs
which shows that \beqs \frac{\pd h}{\pd t} > 0 \eeqs at minimal points. Hence
\beqs h \geq \min\{C_{2}, \min_{\mathbb{S}^{n-1}} h(x,0)\}.\eeqs
The proof of the lemma is completed.
\end{proof}

\begin{corollary}\label{cor3.2} Let $h$ be a smooth solution of $\eqref{seq}$ on $\mathbb{S}^{n-1} \times [0, T)$, and $f , \var, G$ are smooth functions satisfying $\eqref{cond-f}$. Then  we have
\begin{equation}
\begin{split}
|\nabla h(x,t)| \leq C, \quad \forall (x,t) \in \mathbb{S}^{n-1} \times [0, T),\\
\end{split}
\end{equation}
 where $C$ is the same positive constant as in Lemma \ref{lem3.1}.
\end{corollary}
\begin{proof}
By $\eqref{rs}$, we know that
  \beqs \min_{\mathbb{S}^{n-1}} h(x,t) \leq \rho(u,t) \leq
  \max_{\mathbb{S}^{n-1}} h(x,t),\eeqs
and \begin{equation*}
\begin{split}
\rho^{2}  = h^{2} +  |\nabla h|^{2}.
\end{split}
\end{equation*}
These facts together with Lemma \ref{lem3.1} imply the result.
\end{proof}

 Next we will establish the uniform upper and lower bounds for the principal
 curvatures of  the flow $\eqref{floweq}$. We first derive an upper bound for
 the Gauss curvature. For convenience, we write
 \begin{equation*}
 F = f(x)\mathcal{K}(x)\frac{h}{G(\nabla h + hx)}\frac{1}{\varphi(h)}. 
 \end{equation*}
 In this section, we take a local orthonormal frame $\{e_{1}, \cdots, e_{n-1}\}$ on $\mathbb{S}^{n-1} $ such that the standard metric on $\mathbb{S}^{n-1} $ is $\{\delta_{ij}\}$. The local coordinate of a point $x$ in $\mathbb{S}^{n-1}$ is denoted by $x = (x_{1}, \cdots, x_{n-1})$. Double indices
 always mean to sum from $1$ to $n-1$.

By Lemma \ref{lem3.1} and Corollary \ref{cor3.2}, if  $h$ is a smooth
   solution of $\eqref{seq}$ on $\mathbb{S}^{n-1} \times [0, T)$, and $f , \var,
   G$ are smooth functions satisfying $\eqref{cond-f}$, then along the flow for
   $[0, T)$, $\nabla h + hx$ and $h$ are  smooth functions whose ranges are within
   some bounded domain $\Omega_{[0, T)}$  and bounded interval $I_{[0, T)}$
   respectively. 
   Here $\Omega_{[0, T)}$ and $I_{[0, T)}$ depend only on the
   upper and lower bounds of $h$  on $[0, T)$.

\begin{lemma}\label{lem3.3}
Let $h$ be a smooth solution of $\eqref{seq}$ on $\mathbb{S}^{n-1} \times [0, T)$, and $f , \var, G$ are smooth functions satisfying $\eqref{cond-f}$. Then on $\mathbb{S}^{n-1} \times [0, T)$,
\beqs  \mathcal{K}(x,t) \leq C,\eeqs
  where $C$ is a positive constant depending on
  $\|f\|_{C^{0}(\mathbb{S}^{n-1})}$, $\|\var\|_{C^{1}(I_{[0, T)})}$,
  $\|G\|_{C^{1}(\Omega_{[0, T)})}$ and $\|h\|_{C^{1}(\mathbb{S}^{n-1} \times [0, T))}$.
\end{lemma}

\begin{proof}
Consider the following auxiliary function   \beqs Q(x,t)  = \frac{1}{h- \varepsilon _{0}}(F - h),\eeqs
where $ \varepsilon _{0}$ is a positive constant satisfying
\beqs \varepsilon _{0} < \min_{\mathbb{S}^{n-1} \times [0, T)} h(x,t).\eeqs
Recall that $F = \frac{f(x)\mathcal{K}(x)h}{G(\nabla h + hx)\varphi(h)}$, then by Lemma \ref{lem3.1} and Corollary \ref{cor3.2} , the upper bound of $\mathcal{K}(x,t) $  follows from that of $ Q(x,t)$.
Hence we only need to derive the upper bound of $ Q(x,t)$. First we   compute
the evolution equation of  $ Q(x,t)$. Note that
\begin{equation*}
\begin{split}
\nabla_{i} Q
&= \frac{F_{i} - h_{i}}{h- \varepsilon _{0}} - \frac{F-h}{(h- \varepsilon _{0})^{2}}h_{i},\\
 \nabla_{ij} Q
&= \frac{F_{ij} -h_{ij}  }{h- \varepsilon _{0}} - \frac{(F_{i} - h_{i}) h_{j}}{(h- \varepsilon _{0})^{2}} - \frac{(F_{j} - h_{j})h_{i} + (F-h)h_{ij}}{(h- \varepsilon _{0})^{2}} + 2\frac{(F-h)h_{i}h_{j}}{(h- \varepsilon _{0})^{3}}\\
& = \frac{F_{ij} -h_{ij}  }{h- \varepsilon _{0}} - \frac{(F-h)h_{ij}}{(h- \varepsilon _{0})^{2}} - \frac{\nabla_{i}Q h_{j}}{h- \varepsilon _{0}}
- \frac{\nabla_{j}Q h_{i}}{h- \varepsilon _{0}},
\end{split}
\end{equation*}
 and
\beqs \frac{\pd Q}{\pd t} = \frac{F_{t} - h_{t}}{h- \varepsilon _{0}}  + \frac{h_{t}^{2}}{(h- \varepsilon _{0})^{2}}= \frac{F_{t} }{h- \varepsilon _{0}} + Q + Q^{2}.
\eeqs
 The evolution equation of $ Q(x,t)$ is given by
\begin{equation*}
\begin{split}
\frac{\pd Q}{\pd t} - Fb^{ij} \nabla_{ij} Q
&=  \frac{1}{h- \varepsilon _{0}}(F_{t} -Fb^{ij}F_{ij}) + Q + Q^{2}  \\
&\hskip1.1em + (Q+1)\frac{Fb^{ij}h_{ij}}{h- \varepsilon _{0}} +  \frac{ \nabla_{i} Q F{b^{ij}h_{j}}}{h- \varepsilon _{0}} +  \frac{ \nabla_{j} Q F{b^{ij}h_{i}}}{h- \varepsilon _{0}}.
\end{split}
\end{equation*}
Now, we need to compute the evolution equation of $F$.
\begin{equation*}
\begin{split}
F_{t}
&= \frac{hf(x)}{G(\nabla h + hx)\varphi(h)}\frac{\pd \mathcal{K} }{\pd t} +  \mathcal{K}(x,t)f(x)\frac{\pd }{\pd t}(\frac{h}{G(\nabla h + hx)\varphi(h)})\\
&= -Fb^{ij}(h_{ij} + \delta_{ij}h)_{t}  + \mathcal{K}(x,t)f(x)\frac{\pd }{\pd t}(\frac{h}{G(\nabla h + hx)\varphi(h)})\\
&= -Fb^{ij}(h _{t})_{ij} - Fb^{ij}\delta_{ij}h_{t} + \mathcal{K}(x,t)f(x)\frac{\pd }{\pd t}(\frac{h}{G(\nabla h + hx)\varphi(h)})\\
&= -Fb^{ij}(-F + h)_{ij}  - Fb^{ij}\delta_{ij}h_{t} + \mathcal{K}(x,t)f(x)\frac{\pd }{\pd t}(\frac{h}{G(\nabla h + hx)\varphi(h)})\\
&= Fb^{ij}F_{ij} - Fb^{ij}b_{ij} +  F^{2}b^{ij}\delta_{ij} + \mathcal{K}(x,t)f(x)\frac{\pd }{\pd t}(\frac{h}{G(\nabla h + hx)\varphi(h)}),
\end{split}
\end{equation*}
where we have used the fact
\beqn\label{dK}\frac{\pd \mathcal{K}}{\pd b_{ij}} = -\mathcal{K}b^{ij}.\eeqn

Therefore we have
\beqs \frac{\pd F}{\pd t}  -Fb^{ij}\nabla_{ij}F= - F(n-1) +  F^{2}b^{ij}\delta_{ij}+ \mathcal{K}(x,t)f(x)\frac{\pd }{\pd t}(\frac{h}{G(\nabla h + hx)\varphi(h)}).\eeqs
 At a spatial maximal point of $Q(x,t)$, if we take an orthonormal frame such that $b_{ij}$ is diagonal, we have
\begin{equation*}
\begin{split}
 \frac{\pd Q}{\pd t}  & -  b^{ii}F\nabla_{ii}Q \\
&\leq   \frac{1}{h- \varepsilon _{0}}(F_{t} - b^{ii}F\nabla_{ii}F) + Q + Q^{2}
+  \frac{Fb^{ii}h_{ii}}{h- \varepsilon _{0}} +  \frac{(F-h)F{b^{ii}h_{ii}}}{(h- \varepsilon _{0})^{2}}\\
& =  \frac{1}{h- \varepsilon _{0}}(- Fb^{ii}b_{ii}
+ F^{2}b^{ii}\delta_{ii}+ \mathcal{K}(x,t)f(x)\frac{\pd }{\pd t}(\frac{h}{G(\nabla h + hx)\varphi(h)}))  +  Q + Q^{2} \\
&\hskip1.1em +  \frac{Fb^{ii}(b_{ii} - h\delta_{ii})}{h- \varepsilon _{0}} +  \frac{QF{b^{ii}(b_{ii} - h\delta_{ii})}}{h- \varepsilon _{0}}\\
& =  \frac{F^{2}}{h- \varepsilon _{0}}\sum_{i}{b^{ii}}+ Q + Q^{2} + \frac{1}{h- \varepsilon _{0}}\mathcal{K}(x,t)f(x)\frac{\pd }{\pd t}(\frac{h}{G(\nabla h + hx)\varphi(h)})\\
&\hskip1.1em - \frac{hF}{h- \varepsilon _{0}}\sum_{i}{b^{ii}} + \frac{QF(n-1)}{h- \varepsilon _{0}}  - \frac{QFh}{h- \varepsilon _{0}}\sum_{i}{b^{ii}}\\
& \leq  FQ(1- \frac{h}{h- \varepsilon _{0}})\sum_{i}{b^{ii}} + C_{1}Q +  C_{2} Q^{2}  \\
&\hskip1.1em + \frac{1}{h- \varepsilon _{0}}\mathcal{K}(x,t)f(x)\frac{\pd }{\pd t}(\frac{h}{G(\nabla h + hx)\varphi(h)}).
\end{split}
\end{equation*}
From the evolution equation of $\eqref{floweq}$, we have
\begin{equation*}
  \begin{split}
    \frac{\pd G}{\pd t}
    &= \langle \overline{\nabla}G, \pd _{t} X \rangle \\
    &= \langle \overline{\nabla}G, -F\nu + X \rangle \\
    &= -F\langle \overline{\nabla}G,  \nu \rangle  + \langle \overline{\nabla}G,  X \rangle, 
  \end{split}
\end{equation*}
where $\overline{\nabla}$ is the standard metric on $\R^{n}$.

Now, we have
\begin{equation*}
\begin{split}
\frac{\pd }{\pd t}(\frac{h}{G(\nabla h + hx)\varphi(h)}) 
&= \frac{1}{G\varphi} h_{t} - \frac{h}{G^{2}\varphi^{2}}
(\varphi \frac{\pd G}{\pd t} + G \frac{\pd \varphi}{\pd t} )\\
&= \frac{1}{G\varphi} h_{t} - \frac{h\varphi'}{G\varphi^{2}} h_{t} - \frac{h}{G^{2}\varphi} \frac{\pd G}{\pd t} \\
&\leq  cQ(h- \varepsilon _{0}) + c,
\end{split}
\end{equation*}
where  $\varphi' $   denotes $ \frac{\pd \var(s)}{\pd s}$, and $c$ is a positive
constant depending on $\|\var\|_{C^{1}(I_{[0, T)})}$, $\|G\|_{C^{1}(\Omega_{[0, T)})}$ and
$\|h\|_{C^{1}(\mathbb{S}^{n-1} \times [0, T))}$.

For a $Q$ large enough, by Lemma \ref{lem3.1} and Corollary \ref{cor3.2},
it is easy to see
$$ \frac{1}{C_{0}}\mathcal{K}\leq Q \leq C_{0}\mathcal{K},$$
and
\beqs \sum_{i}{b^{ii}} \geq (n-1) \mathcal{K}(x)^{\frac{1}{n-1}}.\eeqs
Hence if we take $ \varepsilon _{0} $ as a positive constant satisfying $ \varepsilon _{0} <  \min_{S^{n-1} \times (0, +\infty)}h(x,t),$ then
for $Q$ large, \beqs \frac{\pd Q}{\pd t}\leq C_{1}Q^{2} (C_{2} - \varepsilon _{0}Q^{\frac{1}{n-1}}) < 0.\eeqs
Now the upper bound of $\mathcal{K}(x,t)$ follows.
\end{proof}

Next we will estimate the lower bound of principal curvatures. It is equivalent
to estimate the upper bound  of the eigenvalues of matrix $\{b_{ij}\}$ w.r.t. $\{\delta_{ij}\}$.

\begin{lemma}\label{lem3.4} Let $h$ be a smooth solution of $\eqref{seq}$ on $\mathbb{S}^{n-1} \times [0, T)$, and $f , \var, G$ are smooth functions satisfying $\eqref{cond-f}$, then the principal curvature $\kappa_{i}$ for $ i = 1, \cdots, n-1$ satisfies
\beqs  \kappa_{i} \geq C,\eeqs
where $C$ is a positive constant depending on  $\|f\|_{C^{2}(\mathbb{S}^{n-1})}$, $\|\var\|_{C^{2}(I_{[0, T)})}$, $\|G\|_{C^{2}(\Omega_{[0, T)})}$ and $\|h\|_{C^{1}(\mathbb{S}^{n-1} \times [0, T))}$.
\end{lemma}

\begin{proof}
Consider the auxiliary function
$$w (x,t) = \log \lambda_{\max}(b_{ij}) - A\log h + B|\nabla h|^{2},$$
where $A,B$ are constants to be determined. $\lambda_{\max} $ is the maximal eigenvalue of $b_{ij}$.

For any fixed $t$, we assume that $\max\limits_{\mathbb{S}^{n-1}}w (x,t)$ is
attained at $p \in \mathbb{S}^{n-1}$. At $p $, we take an orthogonal frame such that
$b_{ij}(p,t)$ is diagonal and $\lambda_{\max} (p,t) = b_{11}(p,t).$

Now we can write $w (x,t)$ as
$$w (x,t) = \log b_{11} - A\log h + B|\nabla h|^{2}.$$

First, we compute the evolution equation of $w$. Note that
\begin{align*}
  \frac{\pd \log{b_{11}}}{\pd t} - b^{ii}F\nabla_{ii}\log{b_{11}}
  &= b^{11}(\frac{\pd b_{11}}{\pd t}  - Fb^{ii}\nabla_{ii}b_{11} ) + Fb^{ii}(b^{11})^{2}(\nabla_{i}b_{11})^{2}, \\
  \frac{\pd \log{h}}{\pd t} - b^{ii}F\nabla_{ii}\log{h}
  &= \frac{1}{h}(\frac{\pd h}{\pd t}  - Fb^{ii}\nabla_{ii}h ) +  \frac{Fb^{ii}h_{i}^{2}}{h^{2}}, \\
  \frac{\pd |\nabla h|^{2}}{\pd t} - b^{ii}F\nabla_{ii}|\nabla h|^{2}
  &= 2h_{k}(\frac{\pd h_{k}}{\pd t}  - Fb^{ii}\nabla_{ii}h_{k} ) - 2Fb^{ii}h_{ii}^{2},
\end{align*}
we have
\begin{equation*}
  \begin{split}
    \frac{\pd w}{\pd t}  -  Fb^{ii}\nabla_{ii} w
    &= b^{11}(\frac{\pd b_{11}}{\pd t}  - Fb^{ii}\nabla_{ii}b_{11} )  + Fb^{ii}(b^{11})^{2}(\nabla_{i}b_{11})^{2} \\
    &\hskip1.1em - \frac{A}{h}(\frac{\pd h}{\pd t}
    - Fb^{ii}\nabla_{ii}h ) -   \frac{Fb^{ii}h_{i}^{2}}{h^{2}} \\
    &\hskip1.1em + 2Bh_{k}(\frac{\pd h_{k}}{\pd t}  - Fb^{ii}\nabla_{ii}h_{k} ) - 2BFb^{ii}h_{ii}^{2}.
  \end{split}
\end{equation*}

Let
\begin{equation*}
M = \log \frac{hf(x)}{G(\nabla h + hx)\varphi(h)},
\end{equation*}
then $$\log F =  \log \mathcal{K} + M.$$
Differentiating the above equation, by \eqref{dK} we have
\begin{equation*}
\frac{\nabla_{k}F}{F}
= \frac{1}{\mathcal{K}}\frac{\pd \mathcal{K}}{\pd b_{ij}} \nabla_{k}b_{ij} +  \nabla_{k}M
= -b^{ij}\nabla_{k}b_{ij} +  M_{k},
\end{equation*}
and
\begin{equation*}
\frac{\nabla_{kl}F}{F}  - \frac{\nabla_{k}F\nabla_{l}F}{F^{2}}
= -b^{ij}\nabla_{kl}b_{ij} + b^{ii}b^{jj}\nabla_{k}b_{ij}\nabla_{l}b_{ij} +  \nabla_{lk}M.
\end{equation*}
From the above results and the evolution equation of $h$, we have
\begin{equation*}
\begin{split}
\frac{\pd h}{\pd t}  - Fb^{ii}\nabla_{ii}h &= -F + h - Fb^{ii}(b_{ii} - \delta_{ii}h)\\
&=-Fn + h + Fh \sum_{i}b^{ii},
\end{split}
\end{equation*}
and
\begin{equation*}
\begin{split}
\frac{\pd h_{k}}{\pd t} - Fb^{ii}\nabla_{ii}h_{k}
&= -F_{k} + h_{k} - Fb^{ii} \nabla_{k}b_{ii} + Fh_{k}\sum_{i}b^{ii}\\
&= -M_{k}F + h_{k} + Fh_{k}\sum_{i}b^{ii}.
\end{split}
\end{equation*}

On the other hand, from the evolution equation of $h$, we have
\begin{equation*}
  \begin{split}
    \frac{\pd h_{kl}}{\pd t}
    & = - \nabla_{kl}F + h_{kl} \\
    &= - \frac{\nabla_{k}F\nabla_{l}F}{F}  + Fb^{ij}\nabla_{kl}b_{ij} - Fb^{ii}b^{jj}\nabla_{k}b_{ij}\nabla_{l}b_{ij} - F \nabla_{lk}M +h_{kl}.
  \end{split}
\end{equation*}
By the Gauss equation, see \cite{Urb.JDG.33-1991.91} for details,
 \beqs\nabla_{kl} h_{ij} = \nabla_{ij} h_{kl} + 2\delta_{kl}h_{ij} - 2\delta_{ij}h_{kl} + \delta_{kj}h_{il} - \delta_{li}h_{kj},\eeqs
or \beqs \nabla_{kl}b_{ij } = \nabla_{ij} b_{kl} + \delta_{kl}h_{ij} - \delta_{ij}h_{kl} + \delta_{kj}h_{il} - \delta_{li}h_{kj}.\eeqs
Then
\begin{equation*}
  \begin{split}
    \frac{\pd h_{kl}}{\pd t}
&= Fb^{ij}\nabla_{ij}h_{kl} + 2\delta_{kl}Fb^{ij}h_{ij} -  Fb^{ij}\delta_{ij}h_{kl} + Fb^{ik}h_{il} - Fb^{jl}h_{kj}\\
&\hskip1.1em - \frac{\nabla_{k}F\nabla_{l}F}{F}    - Fb^{ii}b^{jj}\nabla_{k}b_{ij}\nabla_{l}b_{ij} - F \nabla_{lk}M +h_{kl}.
\end{split}
\end{equation*}
Hence
\begin{equation*}
\begin{split} \frac{\pd b_{kl}}{\pd t}
&= Fb^{ij}\nabla_{ij}b_{kl} + \delta_{kl}Fb^{ij}(b_{ij} -h\delta_{ij} )-  Fb^{ij}\delta_{ij}(b_{kl}- h\delta_{kl})\\
&\hskip1.1em + Fb^{ik}h_{il} - Fb^{jl}h_{kj} + b_{kl} - h\delta_{kl} + (-F + h)\delta_{kl}   \\
&\hskip1.1em- \frac{\nabla_{k}F\nabla_{l}F}{F}- Fb^{ii}b^{jj}\nabla_{k}b_{ij}\nabla_{l}b_{ij} - F \nabla_{lk}M \\
&= Fb^{ij}\nabla_{ij}b_{kl} + \delta_{kl}F(n-2) - Fb^{ij}\delta_{ij}b_{kl}\\
&\hskip1.1em+ Fb^{ik}h_{il} - Fb^{jl}h_{kj} + b_{kl} \\
&\hskip1.1em- \frac{\nabla_{k}F\nabla_{l}F}{F}   - Fb^{ii}b^{jj}\nabla_{k}b_{ij}\nabla_{l}b_{ij} - F \nabla_{lk}M.
\end{split}
\end{equation*}
When $k=l=1$, we  have \begin{equation*}
\begin{split} \frac{\pd b_{11}}{\pd t}
&= Fb^{ii}\nabla_{ii}(b_{11}) + F(n-2) - F\sum_{i}b^{ii}b_{11} + b_{11}\\
&\hskip1.1em - \frac{F_{1}^{2}}{F}   - Fb^{ii}b^{jj}\nabla_{1}(b_{ij})^{2} - F M_{11} .
\end{split}
\end{equation*}

By the above results, we get
\begin{equation*}
  \begin{split}
    \frac{\pd w}{\pd t}
    & -  Fb^{ii}\nabla_{ii} w \\
    &\leq b^{11}F(n-2) + 1 - A - b^{11}F M_{11} + \frac{AFn}{h} - AF\sum_{i}b^{ii}\\
    &\hskip1.1em -2BFh_{k}M_{k} + 2B|\nabla h|^{2} + 2BF|\nabla h|^{2}\sum_{i}b^{ii} - 2BF \sum_{i}b_{ii} + 4BFh(n-1)\\
    &= b^{11}F(n-2) - b^{11}F M_{11}-2BFh_{k}M_{k} - (A-2B|\nabla h|^{2})F\sum_{i}b^{ii}\\
    &\hskip1.1em -(A-1 - 2B |\nabla h|^{2})  - 2BF \sum_{i}b_{ii} + 4BFh(n-1)+ \frac{AFn}{h}.
  \end{split}
\end{equation*}
If we let $A$ satisfy $$ A \geq 2B \max_{\mathbb{S}^{n-1} \times [0, T)}{|\nabla h|^{2}} + 1,$$
then
\begin{equation}\label{eq:5}
  \begin{split}
    \frac{\pd w}{\pd t}  -  Fb^{ii}\nabla_{ii} w  
    &\leq b^{11}F(n-2) - b^{11}F M_{11}-2BFh_{k}M_{k} \\
    &\hskip1.1em - 2BF \sum_{i}b_{ii} + 4BFh(n-1)+ \frac{AFn}{h}.
  \end{split}
\end{equation}

Now, we estimate $  - b^{11}F M_{11}-2BFh_{k}M_{k}. $
Since
\begin{equation*}
  \begin{split}
    M &= \log {\frac{hf(x)}{G(\nabla h + hx)\varphi(h)}} \\
    & = \log {f(x)} + \log {h(x)}
-\log{G(\nabla h + hx)} - \log{\varphi(h)},
  \end{split}
\end{equation*}
then
\begin{equation*}
  \begin{split}
    \nabla_{k}M
    &= \frac{f_{k}}{f} + \frac{h_{k}}{h} - \frac{1}{G} \langle \overline{\nabla} G, \nabla_{k} X\rangle - \frac{\varphi'}{\varphi}h_{k}\\
& = \frac{f_{k}}{f} + \frac{h_{k}}{h} - \frac{\varphi'}{\varphi}h_{k} - \frac{1}{G} \langle \overline{\nabla} G, e_{i} \rangle b_{ki}.
\end{split}
\end{equation*}
And
\begin{equation*}
  \begin{split}
    \nabla_{11}M
    &= \frac{f_{11}}{f} - \frac{f_{1}^{2}}{f^{2}} + \frac{h_{11}}{h} - \frac{h_{1}^{2}}{h^{2}}
    -  \frac{\varphi''h_{1}^{2} + \varphi'h_{11}}{\varphi} + \frac{(\varphi')^{2}h_{1}^{2}}{\varphi^{2}}\\
    &\hskip1.1em + \frac{1}{G^{2}} \langle \overline{\nabla} G, e_{i} \rangle \langle \overline{\nabla} G, e_{j} \rangle b_{1i}b_{1j}
    - \frac{1}{G} \langle \nabla_{1}\overline{\nabla} G, e_{i} \rangle b_{1i} \\
    & \hskip1.1em + \frac{1}{G}\langle \overline{\nabla} G, x\rangle b_{11}
    - \frac{1}{G} \langle \overline{\nabla} G, e_{i} \rangle b_{11i},
  \end{split}
\end{equation*}
where $\varphi'' = \frac{\pd ^{2}\var(h)}{\pd h^{2}}$ and  we have used the Gauss formula on $\uS$
\beqs \nabla_{ij}x = -\delta_{ij}x.\eeqs
From above computations, we obtain
\begin{equation*}
  \begin{split}
    -2Bh_{k}M_{k}
&= -2Bh_{k}(\frac{f_{k}}{f} + \frac{h_{k}}{h} -   \frac{1}{G} \langle \overline{\nabla} G, e_{i} \rangle b_{ki}- \frac{\varphi'}{\varphi}h_{k})\\
&\leq 2B(\frac{|\nabla f||\nabla h|}{f} + |\nabla h|^{2}\frac{\varphi'}{\varphi}  ) + 2B\frac{h_{k}}{G}\langle \overline{\nabla} G, e_{k} \rangle b_{kk}\\
&\leq  c_{1}B + 2B\frac{h_{k}}{G}\langle \overline{\nabla} G, e_{k} \rangle b_{kk},
\end{split}
\end{equation*}
where $c_{1}$ is positive constants depending on  $\|\var\|_{C^1(I_{[0, T)})}$, $\|f\|_{C^1(\uS)}$ and 
$\|h\|_{C^1(\uS \times [0, T))}$.
We also have
\begin{equation*}
\begin{split}  - b^{11} M_{11}
&=  - b^{11} (\frac{f_{11}}{f} - \frac{f_{1}^{2}}{f^{2}}  - \frac{h_{1}^{2}}{h^{2}}  -  \frac{\varphi''h_{1}^{2} }{\varphi} + \frac{(\varphi')^{2}h_{1}^{2}}{\varphi^{2}}) + b^{11}\frac{\varphi'(b_{11} - h)}{\varphi}- b^{11}\frac{b_{11}-h}{h}\\
&\hskip1.1em - \frac{1}{G^{2}} \langle \overline{\nabla} G, e_{1} \rangle ^{2} b_{11} + \frac{1}{G} \langle \nabla_{1}\overline{\nabla} G, e_{i} \rangle  - \frac{1}{G}\langle \overline{\nabla} G, x\rangle  + \frac{1}{G} \langle \overline{\nabla} G, e_{i} \rangle b^{11}b_{11i}\\
&\leq  c_{2}b^{11} + c_{3}  + b_{11}\frac{1}{G}\langle \langle \overline{\nabla}^{2} G, e_{1} \rangle, e_{1} \rangle  + \frac{1}{G} \langle \overline{\nabla} G, e_{i} \rangle b^{11}b_{11i}\\
&\leq c_{2}b^{11} + c_{3}  + c_{4} b_{11} + \frac{1}{G} \langle \overline{\nabla} G, e_{i} \rangle b^{11}b_{11i},
\end{split}
\end{equation*}
where  $c_{2}$ is positive constant depending  on $\|\var\|_{C^2(I_{[0,T)})}$, $\|f\|_{C^{2}(\uS)}$ and $\|h\|_{C^1(\mathbb{S}^{n-1} \times [0, T))}$ and positive constants $c_{3}$ and $c_{4}$ depend  on $\|G\|_{C^{2}(\Omega_{[0,T)})}$.

From the definition of $w$,
\begin{equation*}
\begin{split}  b^{11}\nabla_{l}b_{11}
& = \nabla_{l} w + \frac{A}{h}h_{l} - 2Bh_{lk}h_{k}\\
& =  \nabla_{l} w + \frac{A}{h}h_{l} - 2Bb_{lk}h_{k} + 2Bhh_{l},
\end{split}
\end{equation*}
At point $p$, we have \beqs - b^{11} M_{11}  \leq c_{1}B  + c_{2}b^{11} + c_{4}b_{11} + c_{5} - 2B\frac{h_{k}}{G}\langle \overline{\nabla} G, e_{k} \rangle b_{kk},
\eeqs
where $c_{5}$ depends  on $\|G\|_{C^2(\Omega_{[0,T)})}$ and $\|h\|_{C^{1}(\mathbb{S}^{n-1} \times [0, T))}$.

Now, we have proved that
\begin{equation*}
 - b^{11}F M_{11}-2BFh_{k}M_{k}
\leq  F(c_{1}B  + c_{2}b^{11} + c_{4}b_{11} + c_{5} ).
\end{equation*}
Hence, from \eqref{eq:5} we have
\begin{multline*}
\frac{\pd w}{\pd t}  -  Fb^{ii}\nabla_{ii} w
\leq
F(c_{1}B  + c_{2}b^{11} + c_{4}b_{11} + c_{5} )  - 2BF \sum_{i}b_{ii} + 4BFh(n-1)+ \frac{AFn}{h}.
\end{multline*}
If we take \beqs B > c_{4}, \eeqs
then for $b_{ii}$ large enough, there is
\begin{equation*}
  \begin{split}
    \frac{\pd w}{\pd t}  -  Fb^{ii}\nabla_{ii} w
    &\leq F(c_{1}B  + c_{2}b^{11} + c_{5}  )  - BF \sum_{i}b_{ii} + 4BFh(n-1)+ \frac{AFn}{h} \\
    &< 0,
  \end{split}
\end{equation*}
which implies that
\beqs \frac{\pd w}{\pd t}  \leq 0. \eeqs
Now recalling Lemma \ref{lem3.1} and Corollary \ref{cor3.2}, one can easily see
the conclusion of this lemma holds.
\end{proof}

Now we have proved that the principal curvatures of $M_{t}$ have uniform upper and lower bounds, this
together with Lemma \ref{lem3.1} and Corollary \ref{cor3.2} implies that the
evolution equation \eqref{seq} is uniformly parabolic on any finite time
interval. Thus, the result of \cite{KS.IANSSM.44-1980.161} and the standard
parabolic theory show that the smooth solution of \eqref{seq} exists for all
time. Namely flow \eqref{floweq} has a long-time solution. 
And by these estimates again, a subsequence of $M_t$ converges in $C^\infty$ to
a positive, smooth, uniformly convex hypersurface $M_\infty$ in $\R^n$.
Now to complete the proof of Theorem \ref{thm2}, it remains to check $M_\infty$
satisfies Eq. \eqref{dOMP-f} with $c=1$.

\section{Existence of the solutions to the equation}
In this section, we first complete the proof of Theorem \ref{thm2}, namely we
will prove the support function $h_\infty$ of $M_\infty$ satisfies the following equation:
\beqn \label{ceq} \det{(\nabla^{2}h + hI)}G(\nabla{h}+hx)\var(h) = f(x),\indent x \in \mathbb{S}^{n-1}.\eeqn
To achieve this, we define a functional, which is non-increasing along the flow.

Define \begin{equation*}
\begin{split}
\widetilde{G}(r,u) &= \int^{r}G(s,u)s^{n-1}ds, \\
\widetilde{\varphi}(t) &= \int^{t} \frac{1}{\varphi(s)}ds.
\end{split}
\end{equation*}
It can be checked that flow $\eqref{floweq}$ is the gradient flow of the functional given by
\begin{equation*}
J(M) = \int_{S^{n-1}}\widetilde{\varphi}(h)f(x)dx - \int_{S^{n-1}}\widetilde{G}(\rho,u)du.
\end{equation*}
Here $h$ and $\rho$ are the support function and radial function of $M$ respectively.

Now, we show that the functional $J(M_{t})$ is non-increasing along the flow $\eqref{floweq}$.

\begin{lemma}\label{lem4.1} The  functional $J(M_{t})$ is non-increasing along the flow $\eqref{floweq}$. That is
\beqs \frac{\pd J}{\pd t}  \leq 0,\eeqs
and the equality holds if and only if $M_{t}$ satisfies the elliptic equation $\eqref{ceq}$.
\end{lemma}

\begin{proof}
We begin from the equation
  \begin{equation*}
\frac{\pd J}{\pd t} = \int_{S^{n-1}}\frac{1}{\varphi(h)} f(x)\frac{\pd h}{\pd t}dx -\int_{S^{n-1}}G(\rho,u)\rho^{n-1}\frac{\pd \rho}{\pd t}du.
\end{equation*}
Since $$ \frac{1}{\rho(u,t)}\frac{\pd \rho}{\pd t} =   \frac{1}{h(x,t)}\frac{\pd
  h}{\pd t},$$ 
then by the fact
\begin{equation*}
 \rho^{n} du = \frac{h(x)}{\mathcal{K}(x)}dx, 
\end{equation*}
we have
\begin{equation*}
\begin{split}
\frac{\pd J}{\pd t} & =  \int_{S^{n-1}}(\frac{f(x)}{\varphi(h)} - \frac{G(\rho,u)}{\mathcal{K}(x)} )\frac{\pd h}{\pd t}dx \\
& = - \int_{S^{n-1}}\frac{G(\rho,u)}{\mathcal{K}(x)}h(\frac{f(x)}{\varphi(h)}\frac{\mathcal{K}(x)}{G(\rho,u)} -1 )^{2}dx\\
& \leq 0.
\end{split}
\end{equation*}

The equality holds if and only if
\begin{equation*}
\frac{f(x)}{\varphi(h)}\frac{\mathcal{K}(x)}{G(\rho,u)} =1,
\end{equation*}
which is just the equation $\eqref{ceq}$.
\end{proof}

From Lemma \ref{lem3.1}, Corollary \ref{cor3.2} and the definition of $J(M_t)$,
there exists a positive constant $C$ which is independent of $t$, such that
\begin{equation} \label{eq:4}
  |J(M_t)|\leq C, \quad \forall t>0.
\end{equation}
By the proof of Lemma \ref{lem4.1}, we have
\begin{equation}\label{eq:6}
  J(M_0) -J(M_t)
  = \int_{0}^{t} \dd t \int_{S^{n-1}}\frac{G(X_t)}{\mathcal{K}(x)}h(x,t) \Bigl(
  \frac{f(x)}{\varphi(h)}\frac{\mathcal{K}(x)}{G(X_t)} -1 \Bigr)^{2}\dd x.
\end{equation}
Combining \eqref{eq:4} and \eqref{eq:6}, and recalling
Lemmas \ref{lem3.3} and \ref{lem3.4}, we have that
\begin{equation*}
  \int_{0}^{\infty} \dd t \int_{S^{n-1}}\Bigl(
  \frac{f(x)}{\varphi(h)}\frac{\mathcal{K}(x)}{G(X_t)} -1 \Bigr)^{2} \dd x < \infty.
\end{equation*}
This implies that there exists a subsequence of times $ t_{j} \rightarrow \infty
$ such that
\begin{equation*}
\int_{S^{n-1}}\Bigl(
  \frac{f(x)}{\varphi(h)}\frac{\mathcal{K}(x)}{G(X_{t_j})} -1 \Bigr)^{2} \dd x \to 0 \ \text{ as } t_j\to\infty.
\end{equation*}
Passing to the limit, we obtain
\begin{equation*}
\int_{S^{n-1}}\Bigl(
\frac{f(x)}{\varphi(h_\infty)}\frac{\mathcal{K}(x)}{G(X_\infty)} -1 \Bigr)^{2} \dd x =0,
\end{equation*}
where $h_\infty$ is the support function of $M_\infty$, and $X_\infty=\delbar h_\infty$.
This implies that
\begin{equation*}
\frac{f(x)}{\varphi(h_\infty)}\frac{\mathcal{K}(x)}{G(\delbar h_\infty)} =1, \quad \forall x\in\uS,
\end{equation*}
which is just equation \eqref{ceq}.
The proof of Theorem \ref{thm2} is now completed.

At the same time, for Theorem \ref{thm1}, we have
 showed that for smooth $\var, G, f$, there exists a smooth solution $h$ to
 $\eqref{ceq}$.
 We now begin to prove Theorem \ref{thm1} for general case.
 Note that
\begin{lemma}
  \label{lem000}
For any positive solution $h$ to Eq. \eqref{dOMP-f} with $c=1$,
  we have
  \begin{equation*}
C_1\leq h\leq C_2 \text{ on } \uS,
  \end{equation*}
  where positive constants $C_1$ and $C_2$ depend only on the bounds of $f(x)$
  and \\ $\varphi(s) \max_{|y|=s} G(y) s^{n-1}$. 
\end{lemma}

\begin{proof}{}
  Assume $\max_{\uS} h=h(\bar{x}) =M$. From its equation
\begin{equation*}
  \varphi(h) G(\delbar h) \det(\nabla^2h +hI) =f \text{ on } \uS,
\end{equation*}
we have
\begin{equation*}
  \begin{split}
f(\bar{x})
&\leq \varphi(h(\bar{x})) G(h(\bar{x}) \bar{x}) h(\bar{x})^{n-1} \\
&\leq \varphi(M) \max_{|y|=M}G(y) M^{n-1},
  \end{split}
\end{equation*}
which together with the assumption \eqref{cond-f} implies that
$M\leq C$ for a constant $C$ 
depending only on $f$ and 
$\varphi(s) \max_{|y|=s} G(y) s^{n-1}$.
Similarly, considering the minimum point of $h$, we obtain the lower bound of $h$.
\end{proof}

Now by a standard approximation and the regularity theory of \MA equation, one
can easily see Theorem \ref{thm1} holds.

\vskip3ex
As said in the Introduction, for general $\varphi$ and $G$, there is no
uniqueness for solutions to Eq. \eqref{dOMP-f}.
We here provide a special uniqueness result to end this paper.
A similar result can be found in \cite[Theorem 6.5]{GHXY.2018}.
Here we provide a different proof which is inspired by
\cite[Proposition 2.1]{CW.Adv.205-2006.33}.

\begin{theorem}\label{thm3}
  Assume $G(y)=G(|y|)$.
  If whenever
\begin{equation} \label{eq:2}
  \varphi(\lambda s_1) G(\lambda s_2) \leq \varphi(s_1) G(s_2) \lambda^{1-n}
\end{equation}
holds for some positive $s_1$, $s_2$ and $\lambda$, there must be $\lambda\geq1$.
Then the solution to the following equation
  \begin{equation} \label{eq:111}
  \det{(\nabla^{2}h + hI)}G(\delbar h)\var(h) = f(x)
  \end{equation}
  is unique.
 \end{theorem}

 \begin{proof}
 Let $h_{1}$ and $h_{2}$ be two solutions of \eqref{eq:111}, we first want to
 prove that
 \begin{equation} \label{eq:3}
\max \frac{h_{1}}{h_{2}} \leq 1. 
 \end{equation}

 Suppose \eqref{eq:3} is not true, namely $\max \frac{h_{1}}{h_{2}}>1$.
   We want to derive a contradiction.
   Assume $\frac{h_{1}}{h_{2}}$ attains its maximum at $p_{0} \in \mathbb{S}^{n-1} $.
Then $h_1(p_0) >h_2(p_0)$.
  Let $L = \log\frac{h_{1}}{h_{2}}$, then at $p_{0}$,
 \beqs  0 = \nabla L = \frac{\nabla h_{1}}{ h_{1}} - \frac{\nabla h_{2}}{ h_{2}},\eeqs
 and
 \begin{equation*}
\begin{split} 0 &\geq \nabla^{2} L\\
  & =  \frac{\nabla^{2} h_{1}}{ h_{1}} - \frac{\nabla h_{1} \otimes \nabla h_{1}}{h_{1}^{2}} -
  \frac{\nabla^{2} h_{2}}{ h_{2}} + \frac{\nabla h_{2} \otimes \nabla h_{2}}{h_{2}^{2}}\\
  & =  \frac{\nabla^{2} h_{1}}{ h_{1}} -   \frac{\nabla^{2} h_{2}}{ h_{2}}.\end{split}
\end{equation*}
 By equation \eqref{eq:111}, we have at $p_{0}$
 \begin{equation*}
\begin{split}
 1
& = \frac{\det{(\nabla^{2}h_{2} + h_{2}I)}G(|\nabla h_{2} + h_{2}x|)\var(h_{2})}{\det{(\nabla^{2}h_{1} + h_{1}I)}G(|\nabla h_{1} + h_{1}x|)\var(h_{1})}\\
& =  \frac{h_{2}^{n-1}\det{(\frac{\nabla^{2}h_{2}}{h_{2}} + I)}G(h_{2}(\sqrt{|\frac{\nabla h_{2}}{h_{2}}|^{2} + 1}))\var(h_{2})} {h_{1}^{n-1}\det{(\frac{\nabla^{2}h_{1}}{h_{1}} + I)}G(h_{1}(\sqrt{|\frac{\nabla h_{1}}{h_{1}}|^{2} + 1}))\var(h_{1})}\\
&\geq
\frac{h_{2}^{n-1}G(h_{2}(\sqrt{|\frac{\nabla h_{1}}{h_{1}}|^{2} + 1}))\var(h_{2})}
{h_{1}^{n-1}G(h_{1}(\sqrt{|\frac{\nabla h_{1}}{h_{1}}|^{2} + 1}))\var(h_{1})}.
\end{split}
\end{equation*}
Write $h_2(p_0)=\delta h_1(p_0)$, and
$s=h_{1}(\sqrt{|\frac{\nabla h_{1}}{h_{1}}|^{2} + 1})(p_0)$.
Then the above inequality reads
\begin{equation*}
\frac{\delta^{n-1}G(\delta s)\varphi(h_2)}{G(s)\varphi(h_1)}\leq 1,
\end{equation*}
namely
\begin{equation*}
G(\delta s)\varphi(\delta h_1)\leq G(s)\varphi(h_1) \delta^{1-n}.
\end{equation*}
By our assumption \eqref{eq:2}, we have $\delta\geq1$.
Namely $h_2(p_0)\geq h_1(p_0)$. This is a contradiction. Thus
\eqref{eq:3} holds.

Interchanging $h_1$ and $h_2$, \eqref{eq:3} implies
\begin{equation*}
  \max\frac{h_2}{h_1}\leq1.
\end{equation*}
Combining it with \eqref{eq:3}, we have $h_1\equiv h_2$. The proof is completed. 
\end{proof}






\end{document}